\documentclass[a4paper,12pt]{amsart}
\usepackage{amssymb}
\usepackage{amsmath, amsthm, amscd, amsfonts, amssymb, graphicx, color}
\usepackage{amsmath, amsthm, amscd, amsfonts, amssymb, graphicx, color}

\usepackage[cp1250]{inputenc}
\usepackage[T1]{fontenc}

\newtheorem{theorem}{Theorem}[section]
\newtheorem{lemma}[theorem]{Lemma}
\newtheorem{proposition}[theorem]{Proposition}

\theoremstyle{definition}
\newtheorem{definition}[theorem]{Definition}

\theoremstyle{remark}
\newtheorem{remark}[theorem]{Remark}
\numberwithin{equation}{section}

\begin{document}
	\author{Stefan Ivkovi\'{c} }
	
	\vspace{15pt}
	
	\title{On New Approach to Semi- Fredholm theory in unital $ C^{ *} $ -algebras }

	\maketitle
	\begin{abstract}
		Axiomatic Fredholm theory in unital $ C^{ *} $ -algebras was established by Ke\v{c}ki\'{c} and Lazovi\'{c} in \cite{KL}. Following the pure algebraic approach  by Ke\v{c}ki\'{c} and Lazovi\'{c}, in \cite{BJMA} we extended further this theory to axiomatic semi-Fredholm and semi-Weyl theory in unital $ C^{ *} $ -algebras. However, recently, in \cite{IS10} we developed another approach to axiomatic Fredholm theory in unital $ C^{ *} $ -algebras which is based on the theory of Hilbert modules and which is equivalent to the algebraic approach by Ke\v{c}ki\'{c} and Lazovi\'{c}. In this paper, we extend further this new Hilbert- module approach from Fredholm theory  to semi-Fredholm and semi-Weyl theory in unital $ C^{ *} $ -algebras. Hence, we provide new proofs to the results in \cite{BJMA}. 
	\end{abstract}
	
	\vspace{15pt}
	
	\begin{flushleft}
	\end{flushleft} 
	
	\vspace{15pt}
	
	\begin{flushleft}
	\end{flushleft}
	
	\vspace{30pt}
\section{Introduction}
The Fredholm and semi-Fredholm theory on Hilbert and Banach spaces started by studying the integral equations introduced in the pioneering work by Fredholm in 1903 in \cite{F}. After that, the abstract theory of Fredholm and semi-Fredholm operators on Hilbert and Banach spaces was further developed in numerous papers and books such as \cite{AP1}, \cite{AP2}. 
In addition to classical semi-Fredholm theory on Hilbert and Banach spaces, several generalizations of this theory have been considered. Breuer for example started the development of Fredholm theory in von-Neumann algebras as a generalization of the classical Fredholm theory for operators on Hilbert spaces. In \cite{BR} and \cite{BR2} he introduced the notion of a Fredholm operator in a von Neumann algebra and established its main properties.  On the other hand, Fredholm theory on Hilbert $C^{*}$-modules as another generalization of the classical Fredholm theory on Hilbert spaces was started by Mishchenko and Fomenko. In \cite{MF} they introduced the notion of a Fredholm operator on the standard Hilbert  $C^{*}$-module and proved a generalization in this setting of some of the main results from the classical Fredholm theory. \\ 
The interest for considering these generalizations comes from the theory of pseudo differential operators acting on manifolds. The classical theory can be applied in the case of compact manifolds, but not in the case of non-compact ones. Even operators on Euclidian spaces are hard to study, for example Laplacian is not Fredholm. Kernels and cokernels of many operators are infinite dimensional Banach spaces, however, they may also at the same time be finitely generated Hilbert modules over some appropriate $C^{*}$-algebra. Similarly, orthogonal projections onto kernels and cokernels of many bounded linear operators on Hilbert spaces are not finite rank projections in the classical sense, but they are still finite projections in an appropriate von Neumann algebra. Therefore, many operators that are not semi-Fredholm in the classical sense may become semi-Fredholm in a more general sense if we consider them as operators on Hilbert $C^{*}$-modules or as elements of an appropriate von Neumann algebra. Hence, by studying these generalized semi-Fredholm operators, we get a proper extension of the classical semi-Fredholm theory to new classes of operators.

Now, Ke\v{c}ki\'{c} and Lazovi\'{c} in \cite{KL} established an axiomatic approach to Fredholm theory. They introduced the notion of a finite type element in a unital $C^{*}$-algebra which generalizes the notion of the compact operator on the standard Hilbert $C^{*}$-module and the notion of a finite operator in a properly infinite von Neumann algebra. They also introduced the notion of a Fredholm type element with respect to the ideal of these finite type elements. This notion is at a same time a generalization of the classical Fredholm operator on a Hilbert space, Fredholm $C^{*}$-operator on the standard Hilbert $C^{*}$-module defined by Mishchenko and Fomenko and the Fredholm operator on a properly infinite von Neumann algebra defined by Breuer. The index of this Fredholm type element takes values in the K-group. They showed that the set of Fredholm type elements in a unital $C^{*}$-algebra is open in the norm topology and they proved a generalization of the Atkinson theorem. Moreover, they proved the multiplicativity of the index in the K-group and that the index is invariant under perturbations of Fredholm type elements by finite type elements.\\
 In \cite{BJMA} we went further in this direction and defined semi-Fredholm and semi-Weyl type elements in a unital $C^{*}$-algebra. We investigated and proved several properties of these elements as a generalization of the results from the classical semi-Fredholm and semi-Weyl theory on Hilbert and Banach spaces.
 
 Recently, in \cite{IS10}, we introduced a new approach to axiomatic Fredholm theory in unital $C^{*}$-algebras and we proved that this approach is in fact equivalent to the above mentioned approach developed by Ke\v{c}ki\'{c} and Lazovi\'{c}. In this new approach we use the fact that a unital $C^{*}$-algebra $ \mathcal{A} $ is isometrically isomorphic to the algebra of all $ \mathcal{A} -$ linear operators on $ \mathcal{A} $ when $ \mathcal{A} $ is considered as a Hilbert module over itself. This enables us to apply some known results from operator theory on Hilbert  $C^{*}$-modules, such as the result concerning the complementability of the kernel and the image of a closed range $C^{*}$-operator (for more details, see \cite[Theorem 2.3.3]{MT}) and in that way we bypass several technical lemmas from the paper by Ke\v{c}ki\'{c} and Lazovi\'{c}  \cite{KL} which require long proofs.
 
 \text{ }
 
 The aim of this paper is to obtain an axiomatic semi-Fredholm and semi-Weyl theory in  unital $C^{*}$-algebras based on the approach introduced in \cite{IS10}. While \cite{IS10} deals with axiomatic Fredholm theory based on the above mentioned Hilbert module approach, in this paper we establish semi-Fredholm and semi-Weyl theory in  unital $C^{*}$-algebras following the same approach. Our motivation was to provide new and shorter proofs of the results given in \cite{BJMA}, and it is therefore the main topic and the purpose of this  paper.            

\section{Preliminaries}
Throughout this paper $ \mathcal{A} $ always stands for a unital $ C^{ *} $ -algebra and $B(\mathcal{A} ) $ denotes the set of all $\mathcal{A} $ - linear bounded operators on $ \mathcal{A} $ when $ \mathcal{A} $ is considered as a right Hilbert module over itself. Since $ \mathcal{A}$ is self-dual Hilbert module over itself, by \cite[Proposition 2.5.2]{MT} all operators that belong to $B(\mathcal{A} ) $ are adjointable. Hence, by \cite[Corollary 2.5.3]{MT} the set $B(\mathcal{A} ) $ is a unital $ C^{ *} $ -algebra.\\
Let $V$ be the map from $ \mathcal{A} $ into $B(\mathcal{A} ) $ given by $ V(a) = L_{a} $ for all $ a \in \mathcal{A} $ where $L_{a} $ is the corresponding left multiplier by $a.$ Then  $V$ is an isometric *-homomorphism, and, since $ \mathcal{A} $ is unital, it follows that $V$ is in fact an isomorphism. Thus, $B(\mathcal{A} ) $ can be identified with $ \mathcal{A} $ by considering the left multipliers. \\

We recall now the following definition. 
\begin{definition} \label{d 09}
	\cite[Definition 1.1]{KL} Let $\mathcal{A} $ be an unital $C^{*}$-algebra, and $\mathcal{F} \subseteq \mathcal{A} $ be a subalgebra which satisfies the following conditions:\\
	(i) $\mathcal{F} $ is a selfadjoint ideal in $\mathcal{A} ,$ i.e. for all $a \in \mathcal{A}, b \in \mathcal{F}  $ there holds $ab,ba \in \mathcal{F} ,$ and $a \in \mathcal{F} $ implies $a^{*} \in \mathcal{F} ;$\\
	(ii) There is an approximate unit $p_{\alpha} \in \mathcal{F}$ consisting of projections;\\
	(iii) If $p,q \in \mathcal{F} $ are projections, then there exists $v \in \mathcal{A} ,$ such that $vv^{*}=q $ and $v^{*}v \perp p, $ i.e. $v^{*}v + p $ is a projection as well.\\
	We shall call the elements of such an ideal  \textit{finite type elements}. Henceforward we shall denote this ideal by $\mathcal{F} .$
\end{definition}

Let $V$ be the isometric *-isomorphism given above. If $\mathcal{F} $  is an ideal of finite type elements in $ \mathcal{A} ,$ then it is not hard to see that $ V(\mathcal{F} ) $ is an ideal of finite type elements in $B(\mathcal{A} ),$ so we may identify $\mathcal{F} $ with $ V(\mathcal{F} ) .$\\
We let $ \text{ Proj}(\mathcal{A}) $ denote the set of all orthogonal projections in $ \mathcal{A} ,$ and, similarly, $ \text{ Proj}(\mathcal{F}) $ denotes the set of all orthogonal projections in $ \mathcal{F} .$ Now we recall the notion of Murray-von Neumann equivalence between orthogonal projections in $C^{*}-$ algebras.
\begin{definition} \label{d 10} \cite[Definition 1.2]{KL}
	Let $\mathcal{A} $ be a unital $C^{*}-$ algebra, and let $\mathcal{F} \subseteq \mathcal{A}$ be an ideal of finite type elements. In the set $ \text{ Proj}(\mathcal{A}) $ we define the equivalence relation:  
	$$p \sim q \Leftrightarrow \exists v \in \mathcal{A} \text{ } vv^{*}=p,\text{ } v^{*}v=p, $$
	i.e. Murray - von Neumann equivalence. The set $ S(\mathcal{F})=\text{Proj}(\mathcal{F})\text{ }/\sim $ is a commutative semigroup with respect to addition, and the set, $K(\mathcal{F})=G(S(\mathcal{F})),$ where $G$ denotes the Grothendic functor, is a commutative group.
\end{definition}
\begin{definition}  \label{r10d 1.1} \cite[Definition 6]{BJMA}
	Let $p, q$ be orthogonal projections in $\mathcal{A}.$ We will denote $p \preceq q $ if there exists an orthogonal projection $p^{\prime} $ in $\mathcal{A} $ such that $p^{\prime} \leq q $ and $ p \sim p^{\prime} .$ 
\end{definition}
The following concept will be of crucial importance in this paper.
\begin{definition} \label{d 11} \cite[Definition 2.1]{KL}
	Let $a \in \mathcal{A} $ and $p,q $ be projections in $\mathcal{A} .$  We say that $a$ is invertible up to pair $(p, q)$ if there exists some $b \in \mathcal{A} $  such that 
	$$(1-q)a(1-p)b=1-q, \text{  } b(1-q)a(1-p)=1-p. $$
	We refer to such $b$ as almost inverse of $a,$ or $(p,q)$-inverse of $a.$
\end{definition}
We recall also the following useful technical lemma. 
\begin{lemma} \label{l 02} \cite[Lemma 2]{BJMA}
	Let $a \in \mathcal{A} $ and $p,q,p^{\prime},q^{\prime} $ be projections in $\mathcal{A} .$ Suppose that $p,q,p^{\prime} \in \mathcal{F} .$ If $a$ is invertible up to pair $(p,q) $ and also invertible up to pair $(p^{\prime},q^{\prime}) ,$ then $q^{\prime} \in \mathcal{F} .$ 
	% Recenica 2 */* ====================================================================
	Similarly, if instead of $p,q,p^{\prime} $ we have that $p,q,q^{\prime} \in \mathcal{F} ,$ then we must have that $p^{\prime} \in \mathcal{F}$ as well.
	%====================================================================================
\end{lemma}

Finally, we recall the definition of semi-Fredholm elements in a unital $C^{*}-$ algebra.
\begin{definition} \label{d 12} \cite[Definition 2.2]{KL} \cite[Definition 5]{BJMA}
	Let $\mathcal{F}$ be the ideal of finite type elements in $ \mathcal{A} $ and $a \in \mathcal{A} .$ We say that $a$ is an upper semi-Fredholm element with respect to the ideal $\mathcal{F}$ if $a$ is invertible up to pair of projections $(p, q)$ where $p \in \mathcal{F} .$ Similarly, we say that $a$ is a lower semi-Fredholm element with respect to the ideal $\mathcal{F}$, however in this case we assume that $q \in \mathcal{F} $ (and not $p$).\\
	Finally, we say that $a \in \mathcal{A} $ is of Fredholm type (or abstract Fredholm element) with respect to the ideal $\mathcal{F}$ if there are projections $p,q \in \mathcal{F} $ such that $a$ is invertible up to $(p,q).$ The index of the element $a$ (or abstract index) is the element of the group $K(\mathcal{F}) $ defined by
	$$\text{ind}(a)=([p],[q]) \in K(\mathcal{F}), $$
	or less formally
	$$\text{ind}(a)=[p]-[q]. $$
\end{definition}
Below are some characterizations of semi-Fredholm type elements.
\begin{lemma} \label{l 06} \cite[Lemma 9]{BJMA}
	Let $a \in \mathcal{A} .$ Then the following holds.\\
	1) If $a$ is an upper semi-Fredholm element and  $p, q$ are projections in $\mathcal{A} $ such that $a$ is invertible up to $(p, q)$ and $(1-q) a (1-p)=a,$ then $a$ is an upper semi-Fredholm element if and only if $p \in \mathcal{F} .$\\
	2) If $a$ is a lower semi-Fredholm element and $p, q$ are projections in $\mathcal{A} $ such that $a$ is invertible up to $(p, q)$ and $(1-q) a (1-p)=a,$ then $a$ is a lower semi-Fredholm element if and only if $q\in \mathcal{F} .$
\end{lemma}
\begin{lemma} \label{l 07} \cite[Lemma 10]{BJMA}
	Let $a \in \mathcal{A} .$ Then $a$ is an upper semi-Fredholm element if and only if $a$ is left invertible up to some projection $p \in \mathcal{F} .$ Similarly, $a$ is a lower semi-Fredholm element if and only if $a$ is right invertible up to some projection $q \in \mathcal{F} .$ 
\end{lemma}
Next we recall the notion of a semi-Weyl type element in $\mathcal{A}$ as a generalization of a semi-Weyl operator on a Hilbert space. 

\begin{definition} \label{r10d 1.3} \cite[Definition 7]{BJMA}
	Let $a \in \mathcal{A} .$ We say that $a$ is an upper semi-Weyl type element with respect to the ideal $\mathcal{F}$ if there exist projections $p, q$ in $\mathcal{A} $ such that $p \in \mathcal{F} ,$ $p \preceq q $ and $a$ is invertible up to pair $(p, q).$ Similarly, we say that $a$ is a lower semi-Weyl type element with respect to the ideal $\mathcal{F}$, only in this case we assume that $q \in \mathcal{F} $ and $q \preceq p .$ Finally, we say that $a$ is a Weyl type element with respect to the ideal $\mathcal{F}$ if $a$ is invertible up to pair $(p, q)$ where $p, q $ are projections in $ \mathcal{F} $ and $p \sim q .$
\end{definition}

\begin{remark} \cite[Remark 2]{BJMA}
	We notice that every Weyl type element has index zero. The converse is true if $K(\mathcal{F}) $ satisfies the cancellation property. 
\end{remark}

Set
\begin{flushleft}
	$$\mathcal{K}\Phi_{+} (\mathcal{A})= \lbrace a \in \mathcal{A} \mid  a \text{ is an upper semi-Fredholm type element }  \rbrace ,$$
	$$\mathcal{K}\Phi_{-} (\mathcal{A})= \lbrace a \in \mathcal{A} \mid  a \text{  is a lower semi-Fredholm type element }  \rbrace ,$$
	$$\mathcal{K}\Phi (\mathcal{A})= \lbrace a \in \mathcal{A} \mid  a \text{ is a Fredholm type element }  \rbrace ,$$
	$$\mathcal{K}\Phi_{+}^{-} (\mathcal{A})= \lbrace a \in \mathcal{A} \mid  a \text{  is an upper semi-Weyl type element }  \rbrace ,$$
	$$\mathcal{K}\Phi_{-}^{+} (\mathcal{A})= \lbrace a \in \mathcal{A} \mid a \text{  is a lower semi-Weyl type element }  \rbrace ,$$
	$$\mathcal{K}\Phi_{0} (\mathcal{A})= \lbrace a \in \mathcal{A} \mid  a \text{  is a Weyl type element }  \rbrace .$$
\end{flushleft}
It is understood that we consider a fixed ideal $\mathcal{F}$ of finite type elements. \\
Notice that by definition we have $\mathcal{K}\Phi_{+}^{-} (\mathcal{A}) \subseteq \mathcal{K}\Phi_{+} (\mathcal{A}), \text{ } \mathcal{K}\Phi_{-}^{+} (\mathcal{A}) \subseteq \mathcal{K}\Phi_{-} (\mathcal{A}) $ and $\mathcal{K}\Phi_{0} (\mathcal{A}) \subseteq \mathcal{K}\Phi (\mathcal{A}) .$ \\ 
We recall also some properties of semi-Weyl type elements. 
\begin{proposition} \label{r10p 1.7} \cite[Proposition 17]{BJMA}
	Let $a \in \mathcal{A} .$ Then the following statements hold.\\
	1) $a \in \mathcal{K}\Phi_{+}^{-}(\mathcal{A})  $ if and only if there exist a left invertible element $b \in \mathcal{A} $ and some $ f \in \mathcal{F} $ such that $ a=b+f ,$\\
	2) $a \in \mathcal{K}\Phi_{-}^{+}(\mathcal{A}) $ if and only if there exist a right invertible element $ b \in \mathcal{A} $ and some $ f \in \mathcal{F} $ such that $a=b+f ,$\\
	3) $a \in \mathcal{K}\Phi_{0} (\mathcal{A}) $ if and only if there exist an invertible element $b \in \mathcal{A}  $ and some $ f \in \mathcal{F} $ such that $ a=b+f .$ 
\end{proposition}
\begin{proposition} \label{r10p 1.4} \cite[Proposition 13]{BJMA}
	The sets $\mathcal{K}\Phi_{+}(\mathcal{A}) \setminus \mathcal{K}\Phi_{+}^{-} (\mathcal{A})  ,$ $\mathcal{K}\Phi_{-}(\mathcal{A}) \setminus \mathcal{K}\Phi_{-}^{+} (\mathcal{A})  $ and $\mathcal{K}\Phi(\mathcal{A}) \setminus \mathcal{K}\Phi_{0} (\mathcal{A}) $ are open in the norm topology of $\mathcal{A} .$
\end{proposition} 
\begin{lemma} \label{semi-Weyl-intersect} \cite[Lemma 19]{BJMA}
	Let $a \in \mathcal{K}\Phi_{+}^{-} (\mathcal{A}) \cap \mathcal{K}\Phi_{-}^{+} (\mathcal{A}) \cap \mathcal{K}\Phi (\mathcal{A}) .$ Then there exist projections $p, q$ in $\mathcal{F}$ such that $a$ is invertible up to $(p, q), qa (1-p)=0,$ $p \preceq q $ and $q \preceq p.$
\end{lemma}
In \cite{DDj2} \DJ{}or\dj{}evi\'{c} introduced the notion of a generalized Weyl operator on a Hilbert space, which is a closed range operator whose kernel is isomorphic to the orthogonal complement of its image. He proved in \cite[Theorem 1]{DDj2} that a composition of two generalized Weyl operator is also a generalized Weyl operator if this composition has closed image. The next proposition is an algebraic generalization of this result in the case of unital $C^{*}-$ algebras.
\begin{proposition} \label{R8 p2.23} \cite[Proposition 20]{BJMA}
	Let, $a,b \in \mathcal{A} $ and suppose that there exist projections $p,q,p^{\prime}, q^{\prime}, \tilde{p}, \tilde{q} $ in $ \mathcal{A}$ such that $(1-q)a(1-p)=a ,$ $(1-q^{\prime})b(1-p^{\prime})=b  ,$  $(1-\tilde{q})ba(1-\tilde{p})=ba $ and $a,b,ba $ are invertible up to $(p,q), $  $(p^{\prime},q^{\prime}) $ and $(\tilde{p},\tilde{q}),$ respectively. If $p \sim q $ and $p^{\prime} \sim q^{\prime},$ then  $\tilde{p} \sim\tilde{q} .$ 
\end{proposition}
At the end, we recall the following definition regarding Hilbert modules.
\begin{definition} \label{DP D04} % \textbf{\underline{DP D04}}\\
	\cite[Definition 2.3.1]{MT} A closed submodule $\mathcal{N} $ in a Hilbert $C^{*}$-module $ \mathcal{M}$ is called (topologically) complementable if there exists a closed submodule $\mathcal{L} $ in $\mathcal{M} $ such that  $\mathcal{N}+\mathcal{L}=\mathcal{M} ,\mathcal{N} \cap \mathcal{L}=0.$
\end{definition}	
By the symbol $\tilde{ \oplus} $ \index{$\tilde{ \oplus} $} we denote the direct sum of modules as given in \cite{MT}.

Thus, if $M$ is a Hilbert $C^{*}$-module and $M_{1}, M_{2}$ are two closed submodules of $M,$ we write $M=M_{1} \tilde \oplus M_{2}$ if $M_{1} \cap M_{2}=\lbrace 0 \rbrace$ and $M_{1}+ M_{2}=M.$ If, in addition $M_{1}$  and $M_{2}$ are mutually orthogonal, then we write $M=M_{1} \oplus M_{2}.$
\begin{remark} \label{r10 r01}\cite[Remark 1.4]{IS10}
	If $\sqcap \in B(\mathcal{A})  $ is a (skew ) projection, then, since $Im \sqcap $ is closed, by \cite[Theorem 2.3.3]{MT} we get that $Im \sqcap $ is  complementable. Hence, every closed and complementable submodule $M$ of $\mathcal{A} $ is orthogonally complementable. The corresponding orthogonal projection onto $M$ will be denoted by $ P_{M } .$
\end{remark}

\section{Main results}

Throughout this section we let $ \mathcal{F} $ be an ideal of finite type elements in $ \mathcal{A} $ and $ \lbrace P_{\alpha} \rbrace $ be an approximate unit for $ \mathcal{F} $ consisting of orthogonal projections. As observed in the previous section, $ \mathcal{A} $ can be identified with $ B( \mathcal{A} ) $ by considering the left multipliers on $ \mathcal{A} .$\\
We start with the following definition, which is an extended version of \cite[Definition 2.8]{IS10}.

\begin{definition}
	Let $F \in B(\mathcal{A}) . $  We say that $ F \in \mathcal{M} \mathcal{K} \Phi_{+} (\mathcal{A} ) $ if there exists a decomposition 
	$$ \mathcal{A} = M_{1} \tilde{\oplus} N_{1}  \stackrel{F}{\longrightarrow} M_{2} \tilde{\oplus} N_{2} = \mathcal{A} $$ 
	with respect to which $F$
	has the matrix 
	$ 
	\begin{pmatrix}
		F_{1}  & 0 \\
		0 & F_{4}  
	\end{pmatrix} 
	$  
	where $F_{1} $ is an isomorphism and $ P_{N_{1} } \in \mathcal{F} .$ \\
	Similarly, we say that $ F \in \mathcal{M} \mathcal{K} \Phi_{-} (\mathcal{A} ) $ if all the above conditions hold except that in this case we assume that $ P_{N_{2} } \in \mathcal{F} .$\\
	Finally, we say that $ F \in \mathcal{M} \mathcal{K} \Phi (\mathcal{A} ) $ if the above conditions hold and both  $P_{N_{1} } P_{N_{2} } \in \mathcal{F} .$ In this case, we put then 
	$$ index F = [P_{N_{1} } ] - [ P_{N_{2} } ]$$ in $ K(\mathcal{F} ) .$
\end{definition}
By \cite[Lemma 2.7]{IS10} it follows that $ \mathcal{M} \mathcal{K} \Phi_{+} (\mathcal{A} ) = \mathcal{K} \Phi_{+} (\mathcal{A} ) ,  \mathcal{M} \mathcal{K} \Phi_{-} (\mathcal{A} ) = \mathcal{K} \Phi_{-} (\mathcal{A} ) $ and $ \mathcal{M} \mathcal{K} \Phi (\mathcal{A} ) = \mathcal{K} \Phi (\mathcal{A} ) .$ Here we again identify $ B(\mathcal{A}) $ with $ \mathcal{A} .$\\
We have the following lemma.
\begin{lemma}  \label{R12 L01}
	Let $F \in \mathcal{M} \mathcal{K} \Phi_{+} (\mathcal{A}) $ and suppose that
	
	$$\mathcal{A} = M_{1} \tilde \oplus N_{1} \stackrel{F}{\longrightarrow}  M_{2} \tilde \oplus N_{2}= \mathcal{A} ,$$
	$$\mathcal{A} = M_{1}^{\prime} \tilde \oplus N_{1}^{\prime} \stackrel{F}{\longrightarrow}  M_{2}^{\prime} \tilde \oplus N_{2}^{\prime}= \mathcal{A} ,$$
	
	are two $\mathcal{M} \mathcal{K} \Phi_{+}-$decompositions for $F.$ If $ P_{N_{2}} \in \mathcal{F},$ then $ P_{N_{2}^{\prime}} \in \mathcal{F}.$ 
\end{lemma}

\begin{proof}
	By exactly the same arguments as in the first part of the proof of \cite[Theorem 2.9]{IS10}  we can find a sufficiently large $ \alpha$  and closed submodules $\mathcal{R}$ and $\mathcal{R}^{\prime}$ of $\mathcal{A}$ such that $Im(I-P_{\alpha}) \tilde{\oplus} \mathcal{R} \cong M_{1},$ $Im(I-P_{\alpha}) \tilde{\oplus} \mathcal{R}^{\prime} \cong M_{1}^{\prime}$ and
	
	$$ \mathcal{A}= Im(I-P_{\alpha}) \tilde{\oplus} ( \mathcal{R} \tilde{\oplus} N_{1} ) \stackrel{F}{\longrightarrow}   
	F (  Im(I-P_{\alpha})) \tilde{\oplus} ( F (\mathcal{R}) \tilde{\oplus} N_{2}    ) = \mathcal{A} ,$$
	
	$$ \mathcal{A}= Im(I-P_{\alpha}) \tilde{\oplus} ( \mathcal{R}^{\prime} \tilde{\oplus} N_{1}^{\prime}  ) \stackrel{F}{\longrightarrow}   
	F (  Im(I-P_{\alpha})) \tilde{\oplus} ( F (\mathcal{R}^{\prime} ) \tilde{\oplus} N_{2}^{\prime} ) = \mathcal{A}$$
	
	are two $\mathcal{M} \mathcal{K} \Phi_{+}-$decompositions for $F,$ where $\mathcal{R}  \cong F (\mathcal{R}),$ $\mathcal{R}^{\prime}  \cong F (\mathcal{R}^{\prime})$ and $P_{\mathcal{R}}, P_{\mathcal{R}^{\prime}} \in \mathcal{F}. $ Indeed, we are in the position to apply the same arguments as in the first part of the proof of \cite[Theorem 2.9]{IS10} because $ P_{N_{1} } , P_{N_{1}^{\prime}} \in \mathcal{F} .$ It follows that $F (\mathcal{R}) \tilde{\oplus} N_{2} \cong  F (\mathcal{R}^{\prime}) \tilde{\oplus} N_{2}^{\prime} $ and from \cite[Lemma 2.1]{IS10} we also have that  $P_{F(\mathcal{R})} , P_{F(\mathcal{R}^{\prime})} \in \mathcal{F}.$ Hence, if $P_{N_{2}} \in \mathcal{F},$ from \cite[Lemma 2.4]{IS10} we deduce that $P_{F(\mathcal{R}) \tilde{\oplus} N_{2} } \in \mathcal{F}.$ Since $F(\mathcal{R}) \tilde{\oplus} N_{2} \cong F (\mathcal{R}^{\prime}) \tilde{\oplus} N_{2}^{\prime},$ from \cite[Lemma 2.3]{IS10} we obtain that  $P_{F(\mathcal{R^{\prime}}) \tilde{\oplus} N_{2}^{\prime} } \in \mathcal{F}.$ Therefore, $P_{N_{2}^{\prime}} \in \mathcal{F} $ since $P_{N_{2}^{\prime}} \leq P_{F(\mathcal{R^{\prime}}) \tilde{\oplus} N_{2}^{\prime} } $ and $\mathcal{F}$ is an ideal.
\end{proof}
Again, by \cite[Lemma 2.7]{IS10} it follows that Lemma \ref{R12 L01} is equivalent to Lemma \ref{l 02},
however, we have provided here a different proof which is motivated by the proof of \cite[Lemma 2.16]{IS1}.\\
Next, we provide some characterizations of $ \mathcal{M} \mathcal{K} \Phi_{+} $ and $ \mathcal{M} \mathcal{K} \Phi_{-} - $ operators.
\begin{lemma}\label{closed}
	Let $F \in B(\mathcal{A}) $ and suppose that $Im F$ is closed. Then\\
	a) $ F \in  \mathcal{M} \mathcal{K} \Phi_{+} (\mathcal{A})$ if and only if $P_{\ker F} \in \mathcal{F} ,$\\
	b) $ F \in  \mathcal{M} \mathcal{K} \Phi_{-} (\mathcal{A})$ if and only if $P_{(Im F)^{\perp}} \in \mathcal{F} .$
\end{lemma}

\begin{proof}
	Since $Im F$ is closed, by \cite[Theorem 2.3.3]{MT} we have that $$ \mathcal{A} =( \ker F)^{\perp} \oplus F = Im F \oplus (ImF)^{\perp} .$$ Then we can proceed in exactly the same way as in the proof of \cite[Lemma 3.1.21]{Ths}.
\end{proof}
\begin{lemma}\label{boundedbelow}
	Let $ F \in B
	(\mathcal{A}) .$ Then $F \in  \mathcal{M} \mathcal{K} \Phi_{+} (\mathcal{A}) $ if and only if there exists a complementable submodule $M$ of $\mathcal{A} $ such that $F_{\mid_{M^{\perp}}} $ is bounded below and $ P_{M} \in \mathcal{F}.$
\end{lemma} 
\begin{proof}
	If such $M$ exists, then by exactly the same arguments as in the proof of \cite[Lemma 3.1]{IS1} we can deduce that $F$ has the matrix 
	$\left\lbrack
	\begin{array}{ll}
		F_{1} & F_{2} \\
		0 & F_{4} \\
	\end{array}
	\right \rbrack
	$
	with respect to the decomposition 
	$$\mathcal{A} = M^{\prime} \oplus {{M^{\prime}}^{\bot}} \stackrel{F}{\longrightarrow} F(M^{\prime}) \oplus F( {M^{\prime}})^{\bot}=\mathcal{A} ,$$
	where $F_{1}$ is an isomorphism. Hence, by the method from the proof of \cite[Lemma 2.7.10]{MT} we can construct an isomorphism $U$ such that $F$ has the matrix
	$\left\lbrack
	\begin{array}{ll}
		F_{1} & 0 \\
		0 & \tilde F_{4} \\
	\end{array}
	\right \rbrack
	$ with respect to the decomposition 
	$$\mathcal{A} = M \tilde \oplus \mathcal{U}(M^{\perp}) \stackrel{F}{\longrightarrow}  F(M)  \oplus F(M)^{\perp}= \mathcal{A} .$$ 
	Now, since $\mathcal{U}$ is an isomorphism, by \cite[Lemma 2.1]{IS10}  it follows that $P_{\mathcal{U}(M^{\perp})} \in \mathcal{F} .$ Thus, we have obtained an $ \mathcal{M} \mathcal{K} \Phi_{+}- $ decomposition for $F .$
	
	 On the other hand, if 
	$$\mathcal{A} = M_{1} \tilde \oplus N_{1} \stackrel{F}{\longrightarrow}  M_{2} \tilde \oplus N_{2}= \mathcal{A} $$ 
	is an $\mathcal{M} \mathcal{K} \Phi_{+}  -$ decomposition for $F,$ then by \cite[Lemma 2.5]{IS10} we have that 
	$$\mathcal{A} = N_{1}^{\perp} \oplus N_{1} {\longrightarrow}  F(N_{1}^{\perp}) \tilde \oplus N_{2}= \mathcal{A} $$  
	is also an $\mathcal{M} \mathcal{K} \Phi_{+}  -$decomposition for $F.$ In particular, $F_{\mid_{N_{1}^{\perp}}} $ is bounded below. 
\end{proof}
The following auxiliary technical lemma will enable us to connect Lemma \ref{closed} and Lemma \ref{boundedbelow} with the results from the previous section. 
\begin{lemma}\label{related}
	Let $ F \in B(\mathcal{A}) .$ Then $F$ has closed image if and only if there exist orthogonal projections $P,Q \in B(\mathcal{A}) $ such that $ (I-Q) F (I-P)=F$ and $F$ is invertible up to $(P,Q) .$ In this case, $P=P_{\ker F} $ and $Q=P_{Im F^{\perp}} .$ Consequently,if $N$ is a closed and complementable submodule of $\mathcal{A} ,$ then $F_{\mid_{N^{\perp}}} $ is bounded below if and only if $F$ is left invertible up to $ P_{N} .$
\end{lemma}

\begin{proof}
	If $P, Q$ are orthogonal projection such that $ (I-Q) F (I-P)=F $ then $F$ has the matrix 
	$\left\lbrack
	\begin{array}{ll}
		F_{1} & 0 \\
		0 & 0 \\
	\end{array}
	\right \rbrack
	$ 
	with respect to the decomposition 
	$$\mathcal{A} = Im (I-P)  \oplus Im P \stackrel{F}{\longrightarrow} Im (I-Q)  \oplus Im Q= \mathcal{A}.$$
	If in addition there exists some $D$ such that $DF=I-P ,$ then we must have that $F_{\mid_{Im(I-P)}} $ is bounded below. Indeed, for all $x \in Im (I-P) $ it holds then that $ \parallel x \parallel = \parallel (DF) x \parallel  \leq \parallel D \parallel \parallel Fx \parallel,$ so $\parallel Fx \parallel \geq \dfrac{1}{\parallel D \parallel} \parallel x \parallel.$ Hence, $\ker F = Im P $ in this case.  Finally, if $FD = I-Q ,$ then $Im (I-Q) \subseteq Im F ,$ thus $Im F = Im F_{1} = Im (I-Q) ,$ which proves the implication in one direction.
	
	Conversely, if $Im F$ is closed, then by \cite[Theorem 2.3.3]{MT}
	we get that $Im F$ and $\ker F$ are orthogonally complementable in $\mathcal{A} .$ Thus, $F$ has the matrix 
	$\left\lbrack
	\begin{array}{ll}
		F_{1} & 0 \\
		0 & 0 \\
	\end{array}
	\right \rbrack
	$ 
	with respect to the decomposition 
	$$\mathcal{A} = Im (I-P)  \oplus Im P \stackrel{F}{\longrightarrow} Im (I-Q)  \oplus Im Q= \mathcal{A},$$
	where $P=P_{\ker F} $ and $Q=P_{Im F^{\perp}} $ and by the Banach open mapping theorem it follows that $F_{1}$ is an isomorphism since $Im F$ is closed. Now, it is not hard to deduce that the operator with the matrix 
	$\left\lbrack
	\begin{array}{ll}
		F_{1}^{-1} & 0 \\
		0 & 0 \\
	\end{array}
	\right \rbrack
	$ 
	with respect to the decomposition 
	$$\mathcal{A} = Im (I-Q)  \oplus Im Q {\longrightarrow} Im (I-Q)  \oplus Im P= \mathcal{A},$$ 
	is the desired $(P,Q)-$inverse of $F.$\\
	Finally, if $N$ is a closed and complementable submodule of $\mathcal{A} $ and $F_{\mid_{N^{\perp}}} $ is bounded below, then $ FP_{N^{\perp}} $ has closed image and $ \ker FP_{N^{\perp}} = N .$ By previous arguments, it follows that $F$ is left invertible up to  $P_{N} .$ On the other hand, if $F$ is left invertible up to  $P_{N} ,$ then by definition there exists some $ D \in B(\mathcal{A}) $ such that $ DFP_{N^{\perp}} = P_{N^{\perp}} .$ By the similar arguments as above, one can deduce then that $F_{\mid_{N^{\perp}}} $ must be bounded below in this case.
\end{proof}
It follows now from Lemma \ref{related} that Lemma \ref{closed} is equivalent to Lemma \ref{l 06}, whereas Lemma \ref{boundedbelow} is equivalent to Lemma \ref{l 07}, however, we have provided here different proofs from those given in \cite{BJMA}.\\ 
The next proposition is motivated by \cite[Lemma 3.1.13]{Ths}.
\begin{proposition}
	Let $F \in B(\mathcal{A}) $ and suppose that $\Pi $ is a skew or an orthogonal projection such that $ ( I -\Pi) \in \mathcal{F} .$ Then $F \in  \mathcal{M} \mathcal{K} \Phi (\mathcal{A}),$ if and only if there exists a decomposition 
	$$Im \Pi = M \tilde \oplus N \stackrel{\Pi F}{\longrightarrow}  M^{\prime} \tilde \oplus N^{\prime} = Im \Pi $$
	with respect to which the operator $\Pi F_{\mid_{Im \Pi}} $ has the matrix 
	$
	\left\lbrack
	\begin{array}{ll}
		(\Pi F)_{1} & 0 \\
		0 & (\Pi F)_{4} \\
	\end{array}
	\right \rbrack
	$
	where $	(\Pi F)_{1} $ is an isomorphism and $ P_{N}, P_{N^{\prime}} \in \mathcal{F}.$ Moreover, in this case $$index F=[ P_{N}]-[P_{N^{\prime}}] .$$ 
\end{proposition}

\begin{proof}
	We notice first that if $Im \Pi = M \tilde \oplus N = M^{\prime} \tilde \oplus N^{\prime}$ then clearly $N$ and $ N^{\prime} $ are complementable in $ \mathcal{A} $ because $ \mathcal{A}= Im \Pi \tilde{\oplus} \ker \Pi.$ Hence, by Remark \ref{r10 r01}, $N$ and $N^{\prime} $ are orthogonally complementable in $\mathcal{A} ,$ so $P_{N} $ and $P_{N}^{\prime}  $ are well-defined in this case. \\
	Next, since $ \Pi = I -( I - \Pi)$ and $ ( I -\Pi) \in \mathcal{F} ,$ by \cite[Theorem 2.12]{IS10} it follows that $\Pi \in  \mathcal{M} \mathcal{K} \Phi (\mathcal{A}).$\\ 
	Suppose that $F \in  \mathcal{M} \mathcal{K} \Phi (\mathcal{A}).$ By \cite[Proposition 2.10]{IS10} we deduce that $\Pi F \Pi \in \mathcal{M} \mathcal{K} \Phi (\mathcal{A}).$

	Let 
	$$ \mathcal{A} = M \tilde \oplus N \stackrel{\Pi F \Pi}{\longrightarrow}  M^{\prime} \tilde \oplus N^{\prime} = \mathcal{A} $$
	be an $\mathcal{M} \mathcal{K} \Phi-$decomposition for $\Pi F \Pi.$ Since  $\Pi F \Pi_{\vert_{M}} $ is an isomorphism, it follows that $ \Pi_{\vert_{M}}$ must be bounded below, so $\Pi (M)$ is closed. Furthermore, since $M$ is orthogonally complementable by Remark \ref{r10 r01} and $\Pi$ is adjointable by \cite[Proposition 2.5.2]{MT}, by the same arguments as in the proof of \cite[Lemma 3.1.13]{Ths} we obtain that $\Pi(M) \oplus \tilde{N}=Im \Pi$ for some closed submodule $\tilde{N}.$ Hence, following further the arguments from the proof of \cite[Lemma 3.1.13]{Ths}, we get that $\Pi$ has the matrix 
	$\left\lbrack
		\begin{array}{ll}
		\Pi_{1} & 0 \\
		0 & \Pi_{4} \\
		\end{array}
	\right \rbrack
	$
	 with respect the decomposition
	$$ \mathcal{A} = M \tilde \oplus \mathcal{U}(N) \stackrel{\Pi}{\longrightarrow}  \Pi(M) \tilde \oplus (\tilde{N} \tilde{\oplus} \ker \Pi )= \mathcal{A} $$
	where $\Pi_{1}$ and $\mathcal{U}$ are isomorphisms. Since $P_{N} \in \mathcal{F} $ and $\mathcal{U}$ is an isomorphism, by \cite[Lemma 2.1]{IS10} we have that $P_{\mathcal{U}(N)} \in \mathcal{F}.$ From Lemma \ref{R12 L01} we get that $P_{\tilde{N} \tilde{\oplus} \ker \Pi} \in \mathcal{F} $ because $\Pi \in  \mathcal{M} \mathcal{K} \Phi (\mathcal{A}).$ Then we must have that $P_{\tilde{N}} \in \mathcal{F} $ because $ P_{\tilde{N}} \leq P_{\tilde{N} \tilde{\oplus} \ker \Pi} $ and $ \mathcal{F} $ is an ideal.\\
	 Next, since $M^{\prime}$ is orthogonally complementable by Remark \ref{r10 r01}, then, as $M^{\prime} \subseteq Im \Pi,$ by \cite[Lemma 2.6]{IS3} we deduce that $M \oplus \tilde{N}^{\prime} =Im \Pi$ for some closed submodule $\tilde{N}^{\prime}.$ Thus, we get 
	$$\mathcal{A} = M^{\prime} \tilde \oplus N^{\prime} =  M^{\prime} \tilde \oplus \tilde{N}^{\prime} \tilde{\oplus} \ker \Pi,$$
	which gives $N^{\prime} \cong \tilde{N}^{\prime} \tilde{\oplus} \ker \Pi.$ By \cite[lemma 2.1]{IS10} we have $P_{\tilde{N}^{\prime} \tilde{\oplus} \ker \Pi} \sim P_{N^{\prime}} \in \mathcal{F}.$ Since $P_{\tilde{N}^{\prime}} \leq P_{\tilde{N}^{\prime} \tilde{\oplus} \ker \Pi} , $  we deduce that $P_{\tilde{N}^{\prime}} \in \mathcal{F}.$\\
	By the same arguments as in the proof of \cite[Lemma 3.1.13]{Ths} we obtain that $\Pi F_{\mid_{Im \Pi}} $ has the matrix
	$\left\lbrack
	\begin{array}{ll}
		(\Pi F)_{1} & 0 \\
		0 & (\Pi F)_{4} \\
	\end{array}
	\right \rbrack
	$
	with respect to the decomposition 
	$$Im \Pi = \tilde{\mathcal{U}}(\Pi(M)) \tilde{\oplus} \tilde{\mathcal{U}} (\tilde{N}) \stackrel{\Pi F}{\longrightarrow} 
	 M^{\prime} \tilde \oplus \tilde N^{\prime}= Im \Pi$$
	where $\tilde{\mathcal{U}}$ and $(\Pi F)_{1} $ are isomorphisms. Since $P_{\tilde{N}} \in \mathcal{F} $ and $\tilde{\mathcal{U}}$ is an isomorphism, by \cite[Lemma 2.1]{IS10} we get that $P_{\tilde{\mathcal{U}} (\tilde{N})} \sim P_{\tilde{N}} \in \mathcal{F}. $ This proves the implication in one direction because $P_{\tilde{N}^{\prime}} \in \mathcal{F}, $ also.
 	
 	Let us show now the implication in the opposite direction. If
 	$$Im \Pi = M \tilde{\oplus} N \stackrel{\Pi F}{\longrightarrow}  M^{\prime} \tilde{\oplus} N^{\prime} = Im \Pi $$
 	is a decomposition satisfying the conditions in the lemma, then by the same arguments as in the proof of \cite[Lemma 3.1.13]{Ths} we obtain that $F$ has the matrix 
	 $\left\lbrack
	 \begin{array}{ll}
	 	F_{1} & F_{2} \\
	 	F_{3} & F_{4} \\
	 \end{array}
	 \right \rbrack
	 $
	 with respect to the decomposition
	$$\mathcal{A} = M \tilde \oplus (N \tilde{\oplus} \ker \Pi) \stackrel{F}{\longrightarrow}  M^{\prime} \tilde \oplus (N^{\prime} \tilde{\oplus} \ker \Pi) = \mathcal{A} ,$$ 
	where $F_{1}$ is an isomorphism. Hence, by the method from the proof of \cite[Lemma 2.7.10]{MT}, we get that $F$ has the matrix  
	$\left\lbrack
	\begin{array}{ll}
		\tilde{F_{1}} & 0 \\
		0 & \tilde{F_{4}} \\
	\end{array}
	\right \rbrack
	$
	with respect to the decomposition 
	$$\mathcal{A} = M \tilde{\oplus} \mathcal{U}(N \tilde{\oplus} \ker \Pi )  \stackrel{F}{\longrightarrow} V(M^{\prime})      \tilde{\oplus} (N^{\prime} \tilde{\oplus} \ker \Pi) = \mathcal{A} $$
	where $\tilde{F_{1}},\mathcal{U}$ and V are isomorphisms. Since $\Pi \in  
	\mathcal{M} \mathcal{K} \Phi (\mathcal{A}),$ by Lemma \ref{closed} we must have that $\ker \Pi$ is orthogonally complementable and $P_{\ker \Pi} \in \mathcal{F}.$ Now, since by the assumption we have that $P_{N}, P_{N^{\prime}} \in \mathcal{F},$ by \cite[Lemma 2.4]{IS10} we get that  $P_{N \tilde{\oplus} \ker \Pi}, P_{N^{\prime} \tilde{\oplus} \ker \Pi} \in \mathcal{F}.$ Finally, since $\mathcal{U}$ is an isomorphism, by \cite[Lemma 2.1]{IS10} we deduce that  $P_{\mathcal{U} (N \tilde{\oplus} \ker \Pi)} \in \mathcal{F}.$ This proves the implication in the opposite direction.\\
	Finally, by using \cite[Proposition 2.10]{IS10} instead of \cite[Lemma 2.7.11]{MT} and by applying \cite[Lemma 2.4]{IS10} we can proceed in exactly the same way as in the last part of the proof of \cite[Lemma 3.1.13]{Ths} to deduce the last statement in the proposition regarding index.
\end{proof}
Motivated by \cite[Definition 5.6]{IS1} we introduce now the following definition.

\begin{definition}
	Let $F \in B(\mathcal{A}).$ We say that $F \in \mathcal{M} \mathcal{K} {\Phi_{+}^{-}}^{\prime} (\mathcal{A})$ if there exists an $\mathcal{M} \mathcal{K} \Phi_{+} $-decomposition 
	$$ \mathcal{A} = M_{1} \tilde \oplus N_{1} \stackrel{F}{\longrightarrow} M_{2} \tilde \oplus N_{2}= \mathcal{A}  $$ 
	for $F$ with the property that $N_{1} \preceq N_{2} ,$ that is $ N_{1} $ is isomorphic to a closed submodule of $ N_{2} .$  \\
	We say that $F \in \mathcal{M} \mathcal{K} {\Phi_{-}^{+}}^{\prime} (\mathcal{A})$ if there exists an $\mathcal{M} \mathcal{K} \Phi_{-} $-decomposition 
	$$ \mathcal{A} = M_{1} \tilde \oplus N_{1} \stackrel{F}{\longrightarrow} M_{2} \tilde \oplus N_{2}= \mathcal{A}  $$ 
	for $F$ with the property that $N_{2} \preceq N_{1} .$\\
	Finally, we say that $F \in \mathcal{M} \mathcal{K} \Phi_{0} (\mathcal{A})$ if there exists an $\mathcal{M} \mathcal{K} \Phi $-decomposition 
	$$ \mathcal{A} = M_{1} \tilde \oplus N_{1} \stackrel{F}{\longrightarrow} M_{2} \tilde \oplus N_{2}= \mathcal{A}  $$ 
	for $F$ with the property that $N_{1} \cong N_{2} .$
\end{definition}
\begin{lemma} \label{R12 L02} 
	Let $M$ and $N$ be two closed, complementable submodules of $\mathcal{A}.$ Then $M$ is isomorphic to a closed submodule of $N$ if and only if $ P_{M} \preceq P_{N}.$	
\end{lemma}

\begin{proof}
	Note first that $M$ is orthogonally complementable by Remark \ref{r10 r01}. Let $ N^{\prime} \subseteq N $ such that $M \cong N^{\prime}$ and denote be $\iota$ the isomorphism from M onto $N^{\prime}.$ If $J$ stands for the inclusion from $N^{\prime}$ into $\mathcal{A},$ then $J \iota P_{M}$ is a bounded, $\mathcal{A}-$linear map on $\mathcal{A},$ hence it is adjointable by \cite[Proposition 2.5.2]{MT}. Since $Im J \iota P_{M} = N^{\prime} ,$ which is closed, by \cite[Theorem 2.3.3]{MT} we obtain that $N^{\prime}$ is orthogonally complementable in $\mathcal{A}.$ By \cite[Lemma 2.1]{IS10} we have that $P_{M} \sim P_{N^{\prime}} ,$ and, obviously, $ P_{N^{\prime}} \leq  P_{M^{\prime}} .$The opposite implication also follows from \cite[Lemma 2.1]{IS10}.
\end{proof}
By Lemma \ref{R12 L02} and \cite[Lemma 2.7]{IS10} it follows that $$\mathcal{M} \mathcal{K} {\Phi_{+}^{-}}^{\prime} (\mathcal{A}) = \mathcal{K}\Phi_{+}^{-} (\mathcal{A}),\text{ } \mathcal{M} \mathcal{K} {\Phi_{-}^{+}}^{\prime} (\mathcal{A}) = \mathcal{K}\Phi_{-}^{+} (\mathcal{A}),\text{ } \mathcal{M} \mathcal{K} \Phi_{0} (\mathcal{A}) = \mathcal{K} \Phi_{0} (\mathcal{A}) .$$
Below we obtain some properties of $ \mathcal{M} \mathcal{K} {\Phi_{+}^{-}}^{\prime} $ and $ \mathcal{M} \mathcal{K} {\Phi_{-}^{+}}^{\prime} -$operators. 
\begin{proposition} \label{R12 P01} 
	Let $F \in B(\mathcal{A}).$ Then $F \in \mathcal{M} \mathcal{K} {\Phi_{+}^{-}}^{\prime} (\mathcal{A})$ if and only if there exist some $D \in B(\mathcal{A})$ and $K \in \mathcal{F}$ such that $F=D+K$ and $D$ is bounded below. 
\end{proposition}
\begin{proof}
	We prove first that the set $\mathcal{M} \mathcal{K} {\Phi_{+}^{-}}^{\prime} (\mathcal{A}) $ is invariant under perturbations by finite type elements. Although it has in fact been proved in \cite{BJMA}, we shall provide here another argumentation. Let $F \in \mathcal{M} \mathcal{K} {\Phi_{+}^{-}}^{\prime} (\mathcal{A})$ and suppose that $$ \mathcal{A} = M_{1} \tilde{\oplus} N_{1}  \stackrel{F}{\longrightarrow} M_{2} \tilde{\oplus} N_{2} = \mathcal{A} $$ is an $ \mathcal{M} \mathcal{K} {\Phi_{+}^{-}}^{\prime}- $ decomposition for $F.$ In particular, it is an $ \mathcal{M} \mathcal{K} \Phi_{+}- $ decomposition for $F ,$ hence  $ P_{N_{1} } \in \mathcal{F} .$ Therefore, if $ K \in \mathcal{F}  ,$ then by exactly the same arguments as in the proof \cite[Lemma 2.11]{IS10} we can find a sufficiently large $ \alpha$ and a closed, complementable submodule $ \mathcal{R} \subseteq \mathcal{A} $ such that 	$$ \mathcal{A} = Im (I-P_{\alpha_{1}}) \tilde{\oplus} ( \mathcal{U}(\mathcal{R}) \tilde{\oplus} \mathcal{U}(N_{1}) )
	\stackrel{F+K}{\longrightarrow} V( Im (I-P_{\alpha_{1}}) ) \tilde{\oplus} ( F(\mathcal{R})  \tilde{\oplus} N_{2} ) = \mathcal{A} $$ is an  $ \mathcal{M} \mathcal{K} \Phi_{+}- $ decomposition for the operator $ F+K ,$ where $ \mathcal{R} \cong F(\mathcal{R}) $ and $ \mathcal{U}, V $ are isomorphisms. Now, since $\mathcal{U}(\mathcal{R}) \cong \mathcal{R} \cong F(\mathcal{R}) $ and $ \mathcal{U}(N_{1}) \cong N_{1} \preceq N_{2} ,$ it is not hard to see that $ \mathcal{U}(\mathcal{R}) \tilde{\oplus} \mathcal{U}(N_{1}) $ is isomorphic to a closed submodule of $ F(\mathcal{R})  \tilde{\oplus} N_{2} .$ Hence, we have obtained an $ \mathcal{M} \mathcal{K} {\Phi_{+}^{-}}^{\prime}- $decomposition for $ F+K ,$ which shows that the set $\mathcal{M} \mathcal{K} {\Phi_{+}^{-}}^{\prime} (\mathcal{A}) $ is invariant under perturbations by finite type elements.\\
	Next, if $D$ is bounded below, then $ImD$ is closed, hence by \cite[Theorem 2.3.3]{MT} we have that $ImD$ is orthogonally complementable in $\mathcal{A} .$ It follows that $D \in \mathcal{M} \mathcal{K} {\Phi_{+}^{-}}^{\prime} (\mathcal{A}) ,$ hence, by previous arguments we get that $D+K \in \mathcal{M} \mathcal{K} {\Phi_{+}^{-}}^{\prime} (\mathcal{A}) ,$ which proves the implication in one direction.\\
	
	By combining the proof of Lemma \ref{R12 L02} with the proof of \cite[Theorem 5.10]{IS1} we can prove the implication in the opposite direction.
	\end{proof} 
It follows from Lemma \ref{related} that Proposition \ref{R12 P01} is equivalent to Proposition \ref{r10p 1.7} part 1), however, we have provided here a proof which is different from the proof given in \cite{BJMA}.\\
 Now we will provide another proof of Proposition \ref{r10p 1.4}, which is different from the original proof of this proposition  given in \cite{BJMA}. The new proof builds further on the proof of \cite[Theorem 4.2]{FIL}.

\begin{proposition}  \label{R12 P02} 
	The sets
	$$ \mathcal{M} \mathcal{K} \Phi_{+} (\mathcal{A}) \setminus \mathcal{M} \mathcal{K} {\Phi_{+}^{-}}^{\prime} (\mathcal{A}), \text{ } 
	\mathcal{M} \mathcal{K} \Phi_{-} (\mathcal{A}) \setminus \mathcal{M} \mathcal{K} {\Phi_{-}^{-}}^{\prime} (\mathcal{A}) $$
	and $\mathcal{M} \mathcal{K} \Phi (\mathcal{A})  \setminus \mathcal{M} \mathcal{K} \Phi_{0} (\mathcal{A})  $ are open.
\end{proposition}

\begin{proof}
	Let $ F \in \mathcal{M} \mathcal{K} \Phi_{+} (\mathcal{A}) \setminus \mathcal{M} \mathcal{K} {\Phi_{+}^{-}}^{\prime} (\mathcal{A}).$ As in the proof of \cite[Theorem 4.2]{FIL}, there exists some $ \epsilon \gneq 0 ,$ such that if $D \in B(\mathcal{A}) $ and  $\parallel F-D \parallel < \epsilon ,$ then we can find $\mathcal{M} \mathcal{K} \Phi_{+} - $decompositions 
	$$\mathcal{A} = M_{1} \tilde \oplus N_{1} \stackrel{F}{\longrightarrow}  M_{2} \tilde \oplus N_{2}= \mathcal{A} ,$$
	$$\mathcal{A} = M_{1}^{\prime} \tilde \oplus N_{1}^{\prime} \stackrel{D}{\longrightarrow}  M_{2}^{\prime} \tilde \oplus N_{2}^{\prime}= \mathcal{A} ,$$ 	
	for $F$ and $D,$ respectively, where $M_{1} \cong M_{1}^{\prime},$ $N_{1} \cong N_{1}^{\prime},$  $M_{2} \cong M_{2}^{\prime}$ and $N_{2} \cong N_{2}^{\prime}.$ 
	Assume that $D \in \mathcal{M}\mathcal{K} {\Phi_{+}^{-}}^{\prime} (\mathcal{A})$  and let
	$$\mathcal{A} = M_{1}^{\prime \prime} \tilde \oplus N_{1}^{\prime \prime} \stackrel{D}{\longrightarrow}  M_{2}^{\prime \prime} \tilde \oplus N_{2}^{\prime \prime}= \mathcal{A} ,$$  
	an $\mathcal{M} \mathcal{K} {\Phi_{+}^{-}}^{\prime} -$decomposition for $D.$ By exactly the same arguments as in the first part of the proof of \cite[Theorem 2.9]{IS10} we can find some $\alpha$ such that 
	$$\mathcal{A} = Im (I-P_{\alpha})  \tilde{\oplus} (\mathcal{R}^{\prime} \tilde{\oplus} N_{1}^{\prime} ) \stackrel{D}{\longrightarrow} D(Im (I-P_{\alpha})) \tilde{\oplus} (D(\mathcal{R}^{\prime})  \tilde{\oplus} N_{2}^{\prime} ) =  \mathcal{A}    ,$$ 
	$$\mathcal{A} = Im (I-P_{\alpha})  \tilde{\oplus} (\mathcal{R}^{\prime \prime} \tilde{\oplus} N_{1}^{\prime \prime} ) \stackrel{D}{\longrightarrow} D(Im (I-P_{\alpha})) \tilde{\oplus} (D(\mathcal{R}^{\prime \prime})  \tilde{\oplus} N_{2}^{\prime \prime} ) =  \mathcal{A}    ,$$ 
	are two $\mathcal{M} \mathcal{K} \Phi_{+} - $decompositions for $D$. By the construction, $\mathcal{R}^{\prime} \cong D(\mathcal{R}^{\prime}),$ $\mathcal{R}^{\prime \prime} \cong D(\mathcal{R}^{\prime \prime}),$ $\mathcal{R}^{\prime},$ $\mathcal{R}^{\prime \prime}$ are orthogonally complementable and $P_{\mathcal{R}^{\prime}},P_{\mathcal{R}^{\prime \prime}} \in \mathcal{F}$ . Then we can proceed as in the proof of \cite[Theorem 4.2]{FIL} to deduce that there exists an isomorphism $\mathcal{U}$ on $\mathcal{A}$ such that $F$ has the matrix 
	$\left\lbrack
	\begin{array}{ll}
		F_{1} & 0 \\
		0 & F_{4} \\
	\end{array}
	\right \rbrack
	$ with respect to the decomposition
	
	$$\mathcal{A} = \mathcal{U}(I-P_{\alpha})  \tilde{\oplus} (\mathcal{U}(\mathcal{R}^{\prime}) \tilde{\oplus} N_{1} ) \stackrel{F}{\longrightarrow} F(\mathcal{U} (I-P_{\alpha})) \tilde{\oplus} (F(\mathcal{U}(\mathcal{R}^{\prime}))  \tilde{\oplus} N_{2}) =  \mathcal{A}    $$
	where $F_{1}$ is an isomorphism and $(F(\mathcal{U}(\mathcal{R}^{\prime})) \cong \mathcal{R}^{\prime} .$ In addition, by the construction from the proof of \cite[Theorem 4.2]{FIL} , we have that $\mathcal{U}(\mathcal{R}^{\prime}) \tilde{\oplus} N_{1} \preceq F(\mathcal{U}(\mathcal{R}^{\prime}))  \tilde{\oplus} N_{2} .$  Since $P_{\mathcal{R}^{\prime}}  \in \mathcal{F} $ and $ \mathcal{U} $ is an isomorphism, by \cite[Lemma 2.1]{IS10} we have that $ P_{\mathcal{U}(\mathcal{R}^{\prime})} \in \mathcal{F} .$ Hence, since $ P_{N_{1}} \in \mathcal{F} ,$ by \cite[Lemma 2.4]{IS10} we get that $P_{\mathcal{U}(\mathcal{R}^{\prime} )   \tilde{\oplus} N_{1}} \in \mathcal{F}.$ Thus, we obtain that 
	$F \in  \mathcal{M} \mathcal{K} {\Phi_{+}^{-}}^{\prime} (\mathcal{A}),$ which is a contradiction. Therefore, we must have that $ D \in \mathcal{M} \mathcal{K} \Phi_{+} (\mathcal{A}) \setminus \mathcal{M} \mathcal{K} {\Phi_{+}^{-}}^{\prime} (\mathcal{A}),$ which shows that the set $ \mathcal{M} \mathcal{K} \Phi_{+} (\mathcal{A}) \setminus \mathcal{M} \mathcal{K} {\Phi_{+}^{-}}^{\prime} (\mathcal{A}) $ is open since the argumentation above holds for every  $D \in B(\mathcal{A}) $ with  $\parallel F-D \parallel < \epsilon .$\\
	The proofs of the other statements are similar. 
\end{proof} 	
Next, we will also provide another proof of Lemma \ref{semi-Weyl-intersect}, which builds further on the proof of \cite[Proposition 4.4]{FIL}.
\begin{lemma}  \label{R12 L03} 
	Let $F \in \mathcal{M} \mathcal{K} \Phi (\mathcal{A} ) \cap \mathcal{M} \mathcal{K}\Phi_{+}^{-'}(\mathcal{A}) \cap \mathcal{M} \mathcal{K} \Phi_{-}^{+'}(\mathcal{A}).$ Then there exists an $\mathcal{M} \mathcal{K} \Phi $-decomposition 
	$$ \mathcal{A} = M_{1} \tilde \oplus N_{1} \stackrel{F}{\longrightarrow} M_{2} \tilde \oplus N_{2}= \mathcal{A}  $$ 
	for $F$ with the property that $N_{1} \preceq N_{2} $ and $N_{2} \preceq N_{1}.$
\end{lemma}

\begin{proof}
	If
	$$\mathcal{A} = M_{1} \tilde \oplus N_{1} \stackrel{F}{\longrightarrow}  M_{2} \tilde \oplus N_{2}= \mathcal{A} $$ and $$\mathcal{A} = M_{1}^{\prime} \tilde \oplus N_{1}^{\prime} \stackrel{F}{\longrightarrow}  M_{2}^{\prime} \tilde \oplus N_{2}^{\prime}= \mathcal{A} $$ 
	are an $\mathcal{M} \mathcal{K} {\Phi_{+}^{-}}^{\prime}$ and an $\mathcal{M} \mathcal{K} {\Phi_{-}^{+}}^{\prime}$ decomposition for $F,$ then, by Lemma \ref{R12 L01}, both these decompositions are actually  $\mathcal{M} \mathcal{K} {\Phi}-$decompositions for $F$ since $ F \in \mathcal{M} \mathcal{K} \Phi (\mathcal{A} ) $ by the assumption. As in the first part of the proof of \cite[Theorem 2.9]{IS10},  we find some $\alpha$ such that
	
	$$ \mathcal{A}= Im(I-P_{\alpha}) \tilde{\oplus} ( \mathcal{R} \tilde{\oplus} N_{1} ) \stackrel{F}{\longrightarrow}   
	F (  Im(I-P_{\alpha})) \tilde{\oplus} ( F (\mathcal{R}) \tilde{\oplus} N_{1}    ) = \mathcal{A}$$
	
	$$ \mathcal{A}= Im(I-P_{\alpha}) \tilde{\oplus} ( \mathcal{R}^{\prime} \tilde{\oplus} N_{1}^{\prime}  ) \stackrel{F}{\longrightarrow}   
	F (  Im(I-P_{\alpha})) \tilde{\oplus} ( F (\mathcal{R}^{\prime} ) \tilde{\oplus} N_{2}^{\prime} ) = \mathcal{A},$$
	are two  $\mathcal{M} \mathcal{K} {\Phi}-$decompositions for $F,$ where  $P_{\mathcal{R} },P_{\mathcal{R}^{\prime }} \in \mathcal{F}$ and $\mathcal{R} \cong F(\mathcal{R}) ,$ $\mathcal{R}^{\prime} \cong F(\mathcal{R}^{\prime}) .$
	Then we can proceed as in the proof of \cite[Proposition 4.4]{FIL}.
\end{proof}
Again, by Lemma \ref{R12 L02} and \cite[Lemma 2.7]{IS10} it follows that Lemma \ref{R12 L03} is equivalent to Lemma \ref{semi-Weyl-intersect}.\\  
At the end we recall the statement of \cite[Proposition 5.1.3]{Ths} originally given in \cite{IS5}.
\begin{proposition}
	Let $ F, D \in B(\mathcal{A}) $ such that $F$ and $D$ have closed image. Suppose that $ \ker F \cong (ImF)^{ \perp} $ and $ \ker D \cong (ImD)^{ \perp} .$ If $DF$ has closed image, then $ \ker DF \cong (ImDF)^{ \perp} .$ 
\end{proposition}
From Lemma \ref{related} and \cite[Lemma 2.1]{IS10} it follows that this proposition is in fact equivalent to Proposition \ref{R8 p2.23}.

\bibliographystyle{amsplain}

\end{document}